\documentclass{article}

\usepackage{hyperref}

\usepackage{enumerate}

\usepackage[centertags]{amsmath}
\usepackage{amsfonts}
\usepackage{amssymb}
\usepackage{amsthm}
\usepackage{newlfont}
\usepackage{mathtools}
\usepackage[top=1in, bottom=1in, left=1in, right=1in]{geometry}
\usepackage{tikz}
\usetikzlibrary{positioning}
\tikzset{bluenode/.style={circle,fill=black!30,minimum size=0.4cm,inner sep=0pt},}
\tikzset{rednode/.style={circle,fill=black!100,minimum size=0.4cm,inner sep=0pt},}
\usepackage{subfig}

\allowdisplaybreaks

\theoremstyle{plain}
\newtheorem{theorem}{Theorem}[section]

\newtheorem{lemma}[theorem]{Lemma}

\theoremstyle{definition}

\newtheorem{remark}{Remark}[section]

\newcommand{\R}{\mathbb{R}}

\newcommand{\Hidden}[1]{}

\newcommand{\dmin}{d_{\min}}

\begin{document}

\title{Spectral gap of the largest eigenvalue of the normalized graph Laplacian}
\author{
J\"urgen Jost\footnote{MPI MiS Leipzig, jost@mis.mpg.de}
~~~~~~~~
Raffaella Mulas \footnote{MPI MiS Leipzig, Raffaella.Mulas@mis.mpg.de}
~~~~~~~
Florentin M\"unch\footnote{MPI MiS Leipzig, muench@mis.mpg.de}
}
\date{}
\maketitle

\begin{abstract}
We offer a new method for proving that the maximal eigenvalue of the normalized graph Laplacian of a graph with $n$ vertices is at least $\frac{n+1}{n-1}$ provided the graph is not complete and that equality is attained if and only if the complement graph is a single edge or a complete bipartite graph with both parts of size $\frac{n-1}2$. With the same method, we also prove a new lower bound to the largest eigenvalue in terms of the minimum vertex degree, provided this is at most $\frac{n-1}{2}$.
\end{abstract}


\section{Introduction}
Spectral graph theory investigates the fundamental relationships between geometric properties of a graph and the eigenvalues of the corresponding linear operator. A general overview is given in \cite{chung1997spectral}.
In terms of the largest eigenvalue, the normalized graph Laplacian is particularly interesting as it measures how close a graph is to a bipartite graph. 
In this paper, we are interested in the reverse question, i.e., how far from a bipartite graph a graph can be. This translates to giving lower bounds on the largest eigenvalue.
Several bounds are given in
\cite{li2014bounds}. However, the best known estimate so far is $\lambda_n \geq \frac{n}{n-1}$ which is only attained if the graph is a complete graph \cite{chung1997spectral}. Here, $\lambda_n$ denotes the largest eigenvalue and $n$ is the number of vertices.

Naturally, the question arises what is the optimal a-priori estimate for $\lambda_n$ for non-complete graphs. Das and Sun \cite{extremal} proved that for all non-complete graphs one has
\[
\lambda_n \geq \frac{n+1}{n-1},
\]
with equality if and only if the complement graph is a single edge or a complete bipartite graph with both parts of size $\frac{n-1}2$. Here we offer a new method for proving these results, see the proofs of Theorem~\ref{thm:EigenvalueEstimate} and Theorem~\ref{thm:rigidity}. Furthermore, we use this new method for showing that, for a graph with minimum vertex degree $\dmin\leq \frac{n-1}{2}$,
\begin{equation*}
    \lambda_n\geq 1+\frac{1}{\sqrt{\dmin(n-1-\dmin)}}.
\end{equation*}
\section{Eigenvalue estimate}
A graph $G=(V,E)$ consists of a finite non-empty vertex set $V$ and a symmetric edge relation $E \subset V\times V$ containing no diagonal elements $(v,v)$. We write $v\sim w$ for $(v,w) \in E$. 
Let $G=(V,E)$ be a graph with $n$ vertices. The vertex degree is denoted by $d(v):=|N(v)|$ where
$N(v):=\{w\in V : w\sim v\}$.
We also denote
$N_k(v):= \{w \in V: d(v,w)=k\}$ where $d(v,w) := \inf\{n:v=x_0 \sim \ldots \sim x_n =w\}$ is the combinatorial graph distance. We say $G$ is connected if $d(v,w)<\infty$ for all $v,w \in V$. We write $C(V) = \R^V$ and we denote the positive semidefinite normalized Laplacian by
$L:C(V) \to C(V)$; it is given by
\[
Lf(x) := \frac{1}{d(x)}\sum_{y\sim x}(f(x)-f(y)).
\]
We will always assume that $d(x)\geq 1$ for all $x \in V$ as the Laplacian is not well defined otherwise.
The inner product is given by
\[
\langle f, g \rangle := \sum_x f(x)g(x)d(x).
\]
The operator $L$ is self-adjoint w.r.t. this inner product and the eigenvalues of $L$ are
\[
0=\lambda_1 \leq \lambda_2 \leq \ldots \leq \lambda_n.
\]
In this paper, we are interested in estimating the largest eigenvalue $\lambda_n$ which, by the min-max principle, can be written as
\[
\lambda_n = \sup_{f\in C(V) \setminus \{0\}}\frac{\langle Lf,f \rangle}{\langle f,f \rangle}.
\]
It is well known that 
\[
\frac{n}{n-1} \leq \lambda_n \leq 2
\] 
where the first inequality is an equality only for the complete graph, and the latter inequality is an equality only for bipartite graphs.
The following theorem was established in \cite{extremal} and gives the optimal a-priori lower bound on $\lambda_n$ for all non-complete graphs. In contrast to \cite[Theorem~3.1]{extremal}, our proof methods are completely different and allow for an extension of this estimate in terms of the minimal vertex degree, see Section~\ref{sec:boundViaDegree}.
\begin{theorem}[{\cite[Theorem~3.1]{extremal}}]\label{thm:EigenvalueEstimate}
Let $G=(V,E)$ be a non-complete graph with $n$ vertices.
Then,
\[
\lambda_n \geq \frac{n+1}{n-1}.
\]
\end{theorem}

\begin{proof}
We first assume that $G$ is connected.
Since $G$ is not complete, there exists a vertex $v$ with $d(v) \leq n-2$. As the graph is connected, we have $N_2(v) \neq \emptyset$.
Let $w \in N_2(v)$. Then, $d(w) \leq n-2$ as $v$ and $w$ are not adjacent. Moreover, $N(v) \cap N(w) \neq \emptyset$.

We write $A:=|N(v) \cap N(w)|$.
We aim to find  a function $f$ with $\langle f,Lf \rangle \geq \frac{n+1}{n-1} \langle f,f\rangle$.
To do so, it is convenient to choose $f$ in such a way that $Lf = \frac{n+1}{n-1} f$ in $v$ and $w$.
Particularly, let $f:V\to \R$ be given by
\[
f(x) := \begin{cases}
-1                &: x \in N(v) \cap N(w),\\
\frac{n-1}2 \frac{A}{d(v)} &: x=v,\\
\frac{n-1}2 \frac{A}{d(w)} &: x=w,\\ 
0&:\mbox{otherwise.}
\end{cases}
\]
We observe
\[
d(v) Lf(v)= d(v) f(v) + A = \frac{n+1}2 A
\]
and thus, $Lf(v)=\frac{n+1}{n-1} f(v)$.
Similarly, $Lf(w)=\frac{n+1}{n-1} f(w)$.
We now claim that $-Lf(x) \geq \frac{n+1}{n-1}$ for all $x \in N(v) \cap N(w)$.
We observe $A \geq 1 \vee (d(v)+d(w)+2-n)$ where $\vee$ denotes the maximum, and we calculate
\begin{align}\label{eq:NvNwNxVSA}
-Lf(x) = \frac{d(x) - |N(x)\cap N(v) \cap N(w)| + f(v) + f(w)}{d(x)} 
&\geq
 1 + \frac{1-A +f(v)+f(w)}{d(x)}. 
\end{align}
As $f(v) + f(w) \geq A$, we can use $d(x)\leq n-1$ and continue
\begin{align}\label{eq:dxVSnMinus1}
\frac{1-A +f(v)+f(w)}{d(x)}  &\geq \frac{1-A}{n-1} + \frac{A}{2d(v)} + \frac{A}{2d(w)} \\
&= \frac{1}{n-1} + A \left(\frac{1}{2d(v)} + \frac{1}{2d(w)}- \frac{1}{n-1} \right). \nonumber
\end{align}
Since $d(v) \leq n-2$ and $d(w) \leq n-2$, we see that the term in brackets is positive and thus,
\begin{align}\label{eq:AEstimate}
A \left(\frac{1}{2d(v)} + \frac{1}{2d(w)}- \frac{1}{n-1} \right) \geq \left[1\vee(d(v)+d(w)+2-n) \right] \left(\frac{1}{2d(v)} + \frac{1}{2d(w)}- \frac{1}{n-1} \right).
\end{align} 
We write $D:=(d(v)+d(w))/2$, and by the harmonic-arithmetic mean estimate, we have $\frac{1}{2d(v)} + \frac{1}{2d(w)} \geq \frac 1 D$ and thus,
\begin{align}\label{eq:dvdwVSD}
A \left(\frac{1}{2d(v)} + \frac{1}{2d(w)}- \frac{1}{n-1} \right) \geq \left[1\vee(2D+2-n)\right] \left(\frac{1}{D} - \frac{1}{n-1} \right).
\end{align}
We aim to show that the latter term is at least $\frac{1}{n-1}$ which, by multiplying with $D(n-1)$ and subtracting $D$, is equivalent to
\begin{align}\label{eq:QuadraticPoly}
\left[1\vee(2D+2-n) \right](n-1-D) -D\geq 0.
\end{align}
If $D \leq \frac{n-1}2$, then the maximum equals $1$ and the inequality follows immediately.
If $D \geq \frac{n-1}2$, then we can discard the ``$1\vee$'', and so the left hand side becomes 
 a concave quadratic polynomial in $D$ with its zero points in $D=n-2$ and $D=\frac{n-1}2$.
Thus, the inequality \eqref{eq:QuadraticPoly}   holds true for all $D$ between the zero points. Moreover by assumption, $D$ has to be between the zero points which proves the claim that $-Lf(x) \geq \frac{n+1}{n-1}$ for all $x \in N(v)\cap N(w)$. Particularly, this shows that $fLf\geq \frac{n+1}{n-1}f^2$. Integrating proves the claim of the theorem for all connected graphs.
For non-connected graphs, the smallest connected component has at most $\frac{n}{2}$ vertices. By the standard estimate $\lambda_n \geq \frac{n}{n-1}$ applied to the smallest connected component, and the fact that the right hand side of the estimate is a decreasing function of $n$, we get
\[
\lambda_n  \geq \frac{n/2}{n/2 - 1} = \frac{n}{n-2} > \frac{n+1}{n-1}
\]
which proves the theorem for non-connected graphs.
\end{proof}

\section{Rigidity}
We now prove that Theorem~\ref{thm:EigenvalueEstimate} gives the optimal bound and we characterize equality in the eigenvalue estimate which can be attained only for two different graphs (Figure \ref{fig:gap}). One of the graphs is the complete graph with only one edge removed. The other graph is surprisingly significantly different. It can be seen as two copies of a complete graph which are joined by a single vertex. Again, our proof methods differ widely from \cite{extremal}.
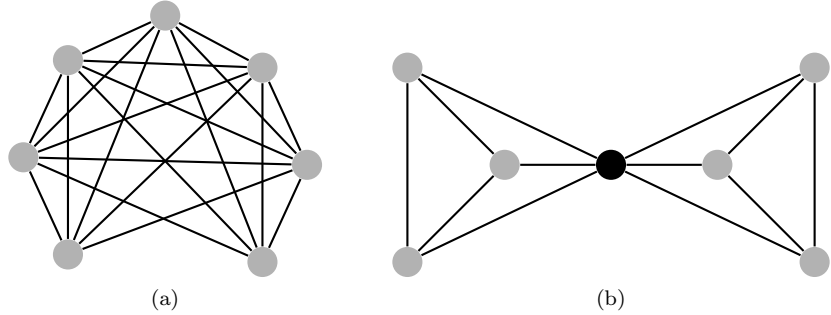
\begin{figure}[ht]
    \centering
     \subfloat[]{{
    \begin{tikzpicture}
			\node[bluenode] (1) {};
			\node[bluenode] (2) [below left = 0.3 cm 											and 1cm of 1]  {};
			\node[bluenode] (3) [below left = 1cm and 											0.3 cm of 2] {};
			\node[bluenode] (4) [below right = 1cm and 										0.3 cm of 3] {};
				\node[bluenode] (5) [below right = 0.4 										cm and 1cm of 1]{};
			\node[bluenode] (6) [below right =1cm and 											0.3 cm of 5] {};
			\node[bluenode] (7) [below left =1cm and 											0.3 cm of 6] {};
			\path[draw,thick]
			(1) edge node {} (2)
			(1) edge node {} (3)
			(1) edge node {} (4)
			(1) edge node {} (5)
			(1) edge node {} (6)
			(1) edge node {} (7)
			(2) edge node {} (3)
			(2) edge node {} (4)
			(2) edge node {} (5)
			(2) edge node {} (6)
			(2) edge node {} (7)
			(3) edge node {} (4)
			(3) edge node {} (5)
			(3) edge node {} (6)
			(3) edge node {} (7)
			(4) edge node {} (5)
			(4) edge node {} (6)
			(5) edge node {} (6)
			(5) edge node {} (7)
			(7) edge node {} (6);
			
			\end{tikzpicture}
					}}
    \qquad
    \subfloat[]{{
    \begin{tikzpicture}
			\node[rednode] (1) {};
			\node[bluenode] (2) [left = of 1]  {};
			\node[bluenode] (3) [below left = of 2] {};
			\node[bluenode] (4) [above left = of 2] {};
				\node[bluenode] (5) [right = of 1]  {};
			\node[bluenode] (6) [below right = of 5] {};
			\node[bluenode] (7) [above right = of 5] {};
			
			\path[draw,thick]
			(1) edge node {} (2)
			(1) edge node {} (3)
			(1) edge node {} (4)
			(1) edge node {} (5)
			(1) edge node {} (6)
			(1) edge node {} (7)
			(2) edge node {} (3)
			(2) edge node {} (4)
			(4) edge node {} (3)
			(5) edge node {} (6)
			(5) edge node {} (7)
			(7) edge node {} (6);
			
			\end{tikzpicture}
				}}	
    \caption{For $n=7$, these are the two graphs in Theorem \ref{thm:rigidity}. The graph on the left is the complete graph $K_7$ with one edge removed. The graph on the right is made by two copies of the complete graph $K_3$, joined by the black vertex in the middle.}
    \label{fig:gap}
\end{figure}
\begin{theorem}[{\cite[Theorem~3.1]{extremal}}]\label{thm:rigidity}
Let $G=(V,E)$ be a graph with $n$ vertices.  T.f.a.e.:
\begin{enumerate}[(i)]
\item
$\lambda_n=\frac{n+1}{n-1}$,
\item
The complement graph of $G$, after removing isolated vertices, is a single edge or a complete bipartite graph with both parts of size $\frac{n-1}2$.
\end{enumerate}
\end{theorem}
\begin{proof}
We first prove $(i)\Rightarrow (ii)$. We first note that $G$ is non-complete but connected by the proof of Theorem~\ref{thm:EigenvalueEstimate}. Thus, all inequalities in the proof of Theorem~\ref{thm:EigenvalueEstimate} must be equalities.
Let $v \not\sim w$ with $N(v) \cap N(w) \neq \emptyset$.
By equality in \eqref{eq:NvNwNxVSA}, all vertices within $N(v) \cap N(w)$ must be adjacent.
By equality in \eqref{eq:dxVSnMinus1}, all vertices of $N(v) \cap N(w)$ must have degree $n-1$. 
By equality in \eqref{eq:AEstimate}, we obtain
\[
|N(v)\cap N(w)| = (d(v)+d(w)+2-n) \vee 1.
\]
By equality in \eqref{eq:dvdwVSD}, we obtain $d(v)=d(w)=D$.
Finally by equality in \eqref{eq:QuadraticPoly}, we see that 
\[
D \in \left\{\frac{n-1}2, n-2 \right\}.
\]
We first show that if $D=n-2$, then the complement graph is a single edge.
If $d(v)=d(w)=D=n-2$, then we get $|N(v)\cap N(w)| = n-2$. Since all vertices within $N(v)\cap N(w)$  are adjacent, we see that the only missing edge is the one from $v$ to $w$ which shows that the complement graph is a single edge.

Now we assume $D=\frac{n-1}{2}$. Then, $A=|N(v) \cap N(w)| = 1$ and we can write $N(v) \cap N(w) = \{x\}$.
We recall that $d(x)=n-1$.
We now specify the parts of the bipartite graph which we want to be the complement graph.
One part is 
\[
P_v := \{v\} \cup N(v)\setminus \{x\}
\]
and similarly, $P_w := \{w\} \cup N(w)\setminus \{x\}$.
Let $\widetilde v \in P_v$ and $\widetilde w \in P_w$.
Then $d(v,\widetilde w) = d(w,\widetilde v) = 2$ as $x$ is adjacent to every other vertex.
By applying the above arguments to the pair $(v,\widetilde w)$, we see that $N(v) \cap N(\widetilde w) = \{x\}$. Particularly, $\widetilde v \not\sim \widetilde w$. Moreover, we have $d(\widetilde w) =d(v)= \frac{n-1}2$ and similarly, $d(\widetilde v) = \frac{n-1}2$. By a counting argument, this shows that 
every $\widetilde v \in P_v$ is adjacent to every vertex not belonging to $P_w$. An analogous statements holds for all  $\widetilde w \in P_w$. Putting everything together, we see that the complement graph of $G$ is precisely the complete bipartite graph with the parts $P_v$ and $P_w$. 
This finishes the case distinction and thus, the proof of the implication $(i)\Rightarrow (ii)$ is complete.

We finally prove $(ii)\Rightarrow (i)$.
We start with the case that the complement graph is the complete bipartite graph. Let the parts be $P$ and $Q$.
Then, $\phi:= 1_P - 1_Q$ is eigenfunction to the eigenvalue $\frac{2}{n-1}$ and every function  orthogonal to $\phi$ and $1$ is eigenfunction to the eigenvalue $\frac{n+1}{n-1}$.

We end with the case that the complement graph is a single edge $(v,w)$. Then, $\phi=1_v - 1_w$ is eigenfunction to eigenvalue $1$, and $\psi = -2 + (n+1)(1_v + 1_w)$ is eigenfunction to eigenvalue $\frac{n+1}{n-1}$. Every function orthogonal to $\phi,\psi$ and $1$ is eigenfunction to the eigenvalue $\frac{n}{n+1}$.

This finishes the proof of $(ii)\Rightarrow (i)$ and thus, the proof of the theorem is complete.
\end{proof}

\begin{remark}
  In  the second equality case in Theorem \ref{thm:rigidity}, for $n>3$, the eigenvalue $\lambda_n$ has multiplicity larger than 1. With the notation of the proof of that theorem, we can take any vertex $v'\in P_v$ and any vertex  $w'\in P_w$ and a function that is $1$ at $v'$ and $w'$, $-1$ at their single joint neighbor $z$, and $0$ everywhere else. For $n=5$, that is when we have two triangles sharing a single vertex $z$. We can also take a function that is $0$ at $z$ and assumes the values $\pm 1$ on the two other vertices of each of the two triangles, to produce other eigenfunctions with eigenvalue $\frac{3}{2}$. 
\end{remark}
\begin{remark}
It is known that, while the eigenvalues of the non-normalized graph Laplacian decrease or stay the same when an edge is removed, the same does not hold for the normalized Laplacian \cite[Remark 9]{interlacing}. The eigenvalues of $L$ may in fact increase when an edge is removed, and in this case \cite[Theorem 8]{interlacing} gives an upper bound to the increase. This is easy to see, for instance, by the fact that $\lambda_n\geq \frac{n}{n-1}$ for every graph, with equality if and only if the graph is complete. Looking at this inequality, however, one may wonder whether the opposite is true for $\lambda_n$, i.e. whether the largest eigenvalue of $L$ increases or stays the same when an edge is removed. The answer is no: Theorem \ref{thm:EigenvalueEstimate} and Theorem \ref{thm:rigidity} offer an example that proves this. In particular, let $G$ be a graph given by two copies of a complete graph which are joined by a single vertex. By Theorem \ref{thm:rigidity}, its largest eigenvalue is $\frac{n+1}{n-1}$. By Theorem \ref{thm:EigenvalueEstimate} and Theorem \ref{thm:rigidity}, any graph given by $G$ with the addition of one edge has largest eigenvalue strictly greater than $\frac{n+1}{n-1}$. Therefore, $\lambda_n$ may as well decrease when an edge is removed.
\end{remark}
\section{Lower bound using the minimum degree}
\label{sec:boundViaDegree}
We now use the same method as that in the proof of Theorem \ref{thm:EigenvalueEstimate} in order to give a new lower bound to the largest eigenvalue in terms of the minimum vertex degree, provided this is at most $\frac{n-1}{2}$. To the best of our knowledge, this is the first known lower bound to $\lambda_n$ in terms of the minimum degree. Li, Guo and Shiu \cite{li2014bounds} proved a bound in terms of the maximum degree. Namely, they have shown that for a graph with $n$ vertices and $m$ edges
\begin{equation*}
    \lambda_n\geq \frac{2m}{2m-\Delta},
\end{equation*}where $\Delta$ is the maximum vertex degree.
\begin{theorem}\label{thm:mindeg}
Let $G=(V,E)$ be a graph with $n$ vertices and let $\dmin$ be the minimum vertex degree of $G$. If $\dmin\leq\frac{n-1}{2}$, then
\begin{equation*}
    \lambda_n\geq 1+\frac{1}{\sqrt{\dmin(n-1-\dmin)}}.
\end{equation*}
\end{theorem}
\begin{proof}
Let
\begin{equation*}
    \psi:=\psi(n,\dmin):=1+\frac{1}{\sqrt{\dmin(n-1-\dmin)}}\qquad\text{and}\qquad \eta:=\frac{1}{\psi-1}.
\end{equation*}We proceed similarly to the proof of Theorem \ref{thm:EigenvalueEstimate}.
We first assume that $G$ is connected. Let $v$ be a vertex of minimum degree, i.e. such that $d(v)=\dmin$. As the graph is connected, we have $N_2(v) \neq \emptyset$.
Let $w \in N_2(v)$. Then, $d(w) \leq n-2$ as $v$ and $w$ are not adjacent. Moreover, $N(v) \cap N(w) \neq \emptyset$.

We write $A:=|N(v) \cap N(w)|$.
We aim to find  a function $f$ with $\langle f,Lf \rangle \geq \psi \langle f,f\rangle$.
To do so, it is convenient to choose $f$ in such a way that $Lf = \psi f$ in $v$ and $w$.
Particularly, let $f:V\to \R$ be given by
\[
f(x) := \begin{cases}
-1                &: x \in N(v) \cap N(w),\\
\eta \cdot\frac{A}{d(v)} &: x=v,\\
\eta \cdot\frac{A}{d(w)} &: x=w,\\ 
0&:\mbox{otherwise.}
\end{cases}
\]
We observe
\[
Lf(v)= (\eta +1)\cdot\frac{A}{d(v)}=\frac{\eta+1}{\eta}\cdot f(v)=\psi f(v)
\]
and similarly, $Lf(w)=\psi f(w)$.
We now claim that $-Lf(x) \geq \psi$ for all $x \in N(v) \cap N(w)$.
We observe $A \geq 1 \vee (d(v)+d(w)+2-n)$ and we calculate
\begin{align}\label{eq:LfDegreeTheorem}
-Lf(x) = \frac{d(x) - |N(x)\cap N(v) \cap N(w)| + f(v) + f(w)}{d(x)} 
&\geq
 1 + \frac{1-A +f(v)+f(w)}{d(x)}. 
\end{align}
In order to proceed, we will use the following lemma which will be proven later independently.
\begin{lemma} We have
\begin{equation*}
    \frac{\eta}{n-1}\cdot \left[1\vee(d(v)+d(w)+2-n) \right] \Biggl(\frac{1}{d(v)}+\frac{1}{d(w)}-\frac{1}{\eta}\Biggr)\geq \frac{1}{\eta}-\frac{1}{n-1},
\end{equation*}
\end{lemma}
Applying the lemma and using $A \geq 1 \vee (d(v)+d(w)+2-n)$ gives
\begin{align*}
0<\frac{1}{\eta}-\frac{1}{n-1} 
&\leq
\frac{\eta}{n-1}\cdot \left[1\vee(d(v)+d(w)+2-n) \right] \Biggl(\frac{1}{d(v)}+\frac{1}{d(w)}-\frac{1}{\eta}\Biggr) \\
&\leq
\frac{\eta A}{n-1}\cdot \Biggl(\frac{1}{d(v)}+\frac{1}{d(w)}-\frac{1}{\eta}\Biggr).
\end{align*}
Moving $\frac{1}{n-1}$ to the right hand side and using $d(x) \leq n-1$ gives
\begin{align*}
0<\frac{1}{\eta}
&\leq
\frac{1}{n-1} +
\frac{\eta A}{n-1}\cdot \Biggl(\frac{1}{d(v)}+\frac{1}{d(w)}-\frac{1}{\eta}\Biggr)\\
&\leq 
\frac{1}{d(x)} +
\frac{\eta A}{d(x)}\cdot \Biggl(\frac{1}{d(v)}+\frac{1}{d(w)}-\frac{1}{\eta}\Biggr)\\
&=\frac{1-A +f(v)+f(w)}{d(x)}\\
&\leq -Lf(x)-1
\end{align*}
where we used \eqref{eq:LfDegreeTheorem} in the last estimate.
Thus, $-Lf(x) \geq 1+\frac{1}{\eta} = \psi$ for all $x \in N(v) \cap N(w)$.
Integrating gives $\langle Lf, f  \rangle \geq \psi \langle f, f \rangle$ which proves the theorem for all connected graphs. For non-connected graph, we apply the theorem for the connected component containing $v$ and use that $\psi$ is decreasing in $n$.
The proof of the theorem is now complete up to the proof of the lemma. 
\end{proof}

\begin{proof}[Proof of the lemma]
We consider two cases.
\begin{enumerate}
    \item Case 1: $d(v)+d(w)\leq n-1$. Then,
    \begin{equation*}
        1\vee(d(v)+d(w)+2-n)=1
    \end{equation*}and \begin{equation*}
        \frac{1}{d(w)}\geq \frac{1}{n-1-d(v)}.
    \end{equation*}Therefore
   \begin{align*}
      \frac{\eta}{n-1}\cdot\Biggl(\frac{1}{d(v)}+\frac{1}{d(w)}-\frac{1}{\eta}\Biggr)\geq \frac{\eta}{n-1}\cdot\Biggl(\frac{1}{d(v)}+\frac{1}{n-1-d(v)}-\frac{1}{\eta}\Biggr).
   \end{align*}Now, we have that
   \begin{align*}
     &\frac{\eta}{n-1}\cdot\Biggl(\frac{1}{d(v)}+\frac{1}{n-1-d(v)}-\frac{1}{\eta}\Biggr)\geq \frac{1}{\eta}-\frac{1}{n-1}\\
      \iff\\
      &\frac{\eta}{n-1}\cdot\Biggl(\frac{1}{d(v)}+\frac{1}{n-1-d(v)}\Biggr)\geq \frac{1}{\eta}\\
      \iff \\
      &\frac{1}{n-1}\cdot\Biggl(\frac{1}{d(v)}+\frac{1}{n-1-d(v)}\Biggr)\geq \frac{1}{\eta^2}.
   \end{align*}This is true by definition of $\eta$ and it is actually an equality.
   \item Case 2: $d(v)+d(w)>n-1$. Then,
    \begin{equation*}
        1\vee(d(v)+d(w)+2-n)=d(v)+d(w)+2-n\geq 2.
    \end{equation*}Therefore, it suffices to prove that
    \begin{equation*}
        \frac{2\eta}{n-1}\cdot \Biggl(\frac{1}{d(v)}+\frac{1}{d(w)}-\frac{1}{\eta}\Biggr)\geq \frac{1}{\eta}-\frac{1}{n-1},
    \end{equation*}i.e. that
    \begin{equation*}
        \frac{2\eta}{n-1}\cdot \Biggl(\frac{1}{d(v)}+\frac{1}{d(w)}\Biggr)\geq \frac{1}{\eta}+\frac{1}{n-1},
    \end{equation*}that can be re-written as
    \begin{equation*}
        \frac{2}{n-1}\cdot \Biggl(\frac{1}{d(v)}+\frac{1}{d(w)}\Biggr)-\frac{1}{\eta^2}\geq \frac{1}{(n-1)\eta}.
    \end{equation*}
In order to prove it, we use the fact that\begin{equation*}
\frac{1}{\eta^2}=\frac{1}{n-1}\cdot\Biggl(\frac{1}{d(v)}+\frac{1}{n-1-d(v)}\Biggr).
\end{equation*}This implies that
\begin{align*}
&\frac{2}{n-1}\cdot \Biggl(\frac{1}{d(v)}+\frac{1}{d(w)}\Biggr)-\frac{1}{\eta^2}\\
 \geq&\frac{2}{n-1}\cdot \Biggl(\frac{1}{d(v)}+\frac{1}{n-1}\Biggr)-\frac{1}{n-1}\cdot\Biggl(\frac{1}{d(v)}+\frac{1}{n-1-d(v)}\Biggr)\\
    =&\frac{1}{n-1}\cdot\Biggl(\frac{1}{d(v)}+\frac{2}{n-1}-\frac{1}{n-1-d(v)}\Biggr)\\
    =&\frac{1}{n-1}\cdot\Biggl(\frac{(n-1)(n-1-d(v))+2d(v)(n-1-d(v))-d(v)(n-1)}{d(v)(n-1)(n-1-d(v))}\Biggr)\\
    =&\frac{(n-1)^2-d(v)(n-1)+2d(v)(n-1)-2d(v)^2-d(v)(n-1)}{d(v)(n-1)^2(n-1-d(v))}\\
    =&\frac{(n-1)^2-2d(v)^2}{d(v)(n-1)^2(n-1-d(v))}.
\end{align*}Therefore, the inequality that we want to prove becomes
\begin{align*}
&\frac{(n-1)^2-2d(v)^2}{d(v)(n-1)^2(n-1-d(v))}\geq \frac{1}{(n-1)\eta}\\
\iff\\
    &\frac{(n-1)^2-2d(v)^2}{d(v)(n-1)(n-1-d(v))}\geq \frac{1}{\eta} = \frac 1 {\sqrt{d(v)(n-1-d(v))}}\\
    \iff\\
    & (n-1)^2-2d(v)^2\geq (n-1)\sqrt{(n-1-d(v))\cdot d(v)}.
\end{align*}Now, since we are assuming $d(v)\leq \frac{n-1}{2}$, 
\begin{equation*}
    (n-1)^2-2d(v)^2\geq \frac{(n-1)^2}{2}.
\end{equation*}Also, by applying $\sqrt{ab}\leq\frac{a+b}{2}$,
\begin{equation*}
    (n-1)\sqrt{(n-1-d(v))\cdot d(v)}\leq (n-1)\frac{n-1}{2}=\frac{(n-1)^2}{2}.
\end{equation*}Therefore,
\begin{equation*}
    (n-1)^2-2d(v)^2\geq (n-1)\sqrt{(n-1-d(v))\cdot d(v)}.
\end{equation*}
\end{enumerate}
Thus, the proof of the lemma is complete.
\end{proof}

\begin{remark}In the particular case of $\dmin=\frac{n-1}{2}$, Theorem \ref{thm:mindeg} tells us that \begin{equation*}\lambda_n\geq \frac{n+1}{n-1}.\end{equation*}By the second part of Theorem \ref{thm:rigidity} we know that this inequality is sharp.\end{remark}
	\bibliographystyle{alpha}
	\bibliography{Gap2020_01_07}	

\end{document}